\documentclass[10pt]{article}
\usepackage{amsmath,amsthm,amsfonts,amssymb,amscd, amsxtra}
\usepackage[margin=2.5cm,nohead]{geometry}

\theoremstyle{plain}
\newtheorem{theorem}{Theorem}[section]
\newtheorem{lemma}{Lemma}[section]
\newtheorem{definition}{Definition}[section]

\newtheorem{proposition}{Proposition}[section]

\newtheorem{remark}{Remark}[section]
\newtheorem{example}{Example}[section]

\newcommand{\beq}{\begin{equation}}
\newcommand{\eeq}{\end{equation}}
\newcommand{\beqa}{\begin{eqnarray}}
\newcommand{\eeqa}{\end{eqnarray}}
\newcommand{\beqas}{\begin{eqnarray*}}
\newcommand{\eeqas}{\end{eqnarray*}}
\newenvironment{algorithm1}{{\noindent\bf Algorithm 1}}{\hfill}

\def\min{\operatorname{min}}
\def\max{\operatorname{max}}

\makeatletter
\renewcommand*{\@biblabel}[1]{\hfill#1.}
\makeatother

\begin{document}

\title{Local convergence of an inexact proximal algorithm for weakly convex functions}

\author{
V. L. Sousa Junior \thanks{Universidade Federal do Cariri, Juazeiro do Norte-CE, Brazil (Email: {\tt valdines.leite@ufca.edu.br}) - {\bf Corresponding author}}
\and
L. V. Meireles
\thanks{Instituto Federal Goiano, Posse-GO, Brazil (Email: {\tt lucas.vidal@ifgoiano.edu.br}).}\and
S. C. Lima\thanks{Universidade Federal do Oeste da Bahia, Barreiras-BA, Brazil  (Email: {\tt samara.lima@ufob.edu.br}).}\and
G. N. Silva  \thanks{Universidade Federal do Piauí, Teresina-PI, Brazil  (Email: {\tt gilson.silva@ufpi.edu.br}).}
}
\date{}
\maketitle
\centerline{\today}
\vspace{.2cm}

\noindent
{\bf Abstract}
Since introduced by Martinet and Rockafellar, the proximal point algorithm was generalized in many fruitful directions. More recently, in 2002, Pennanen studied the proximal point algorithm without monotonicity. A year later, Iusem and Svaiter joined Pennanen to present inexact variants of the method, again without monotonicity. Building on the foundation laid by these two prior works, we propose a variant of the proximal point algorithm designed specifically for weakly convex functions. Our motivation for introducing this inexact algorithm is to increase its versatility and applicability in a broader range of scenarios in optimization and introduce a more adaptable version of the method for typical generalizations. Our study relies heavily on the Moreau envelope, a well-known mathematical tool used to analyze the behavior of the proximal operator. By leveraging the properties of the Moreau envelope, we are able to prove that the proximal algorithm converges in local contexts. Moreover, we present a complexity result to determine the practical feasibility of  the proximal algorithm.




\section{Introduction}

The proximal point algorithm, introduced by Martinet \cite{Martinet1970} and Rockafellar \cite{Rockafellar1976}, has undergone several generalizations beyond convex optimization since its inception. These generalizations have been applied to a variety of situations and have found use in nonconvex programming. Various papers have explored the convergence of the algorithm for non-convex functions, including quasi-convex, pseudo-convex, and difference of convex functions. Kaplan and Tichatschke \cite{Kaplan1998} provides a comprehensive review of the literature on the proximal point algorithm for nonconvex minimization and proposes new conditions for convergence in a special class of nonconvex functions.

Intending to continue this path, we will study the proximal algorithm for weakly convex functions, a class that includes convex functions and smooth functions with Lipschitz continuous gradient, and the composite function class. This class was introduced by Nurminskii \cite{Nurminskii1973} and has been widely used in optimization due to its applicability to a variety of problems. Weakly convex functions are particularly useful for analyzing nonconvex optimization, as they exhibit some convex-like properties that make them amenable to certain optimization techniques. For a more detailed explanation, see \cite{Drusvyatskiy2019}.

The key construction in our analysis is to use the well known \textit{Moreau envelope}. The Moreau envelope of a convex function is always convex, this is a well-known fact in variational analysis and optimization. Notably, the Moreau envelope can also be convex even if its associated function is not (see Example \ref{example1}). This property plays a vital role in our convergence analysis.
The main result shows that the local convexity of the Moreau envelope is sufficient to guarantee the local convergence of the proximal point algorithm for weakly convex functions. Our approach builds upon recent work by Pennanen \cite{Pennanen2002} and Iusem \cite{IusemPennanen2003}, who identified conditions for ensuring the local convergence of the proximal point algorithm for a nonmonotone operator. They observed that the Yosida regularization, which is closely related to the Moreau envelope, can be locally monotone even if its associated operator is not. Furthermore, the local maximal monotonicity of the Yosida regularization is enough to ensure the local convergence of the proximal point algorithm. 

Our main objective is to present a version of the proximal point algorithm that is applicable to weakly convex functions. In our analysis, we use the Moreau envelope as a substitute for the Yosida regularization that was utilized in \cite{IusemPennanen2003,Pennanen2002}. One of the primary challenges in our analysis is to demonstrate the existence of the iterates of the method. In \cite{IusemPennanen2003}, Minty's theorem was used to address this issue. However, as we are not working with a point-set operator, we present a different proof of this result using the Banach fixed-point theorem. While the algorithm we present here is similar to the one presented in \cite{IusemPennanen2003}, it is more suitable for future generalizations in scalar optimization or even in a multi-objective scenario, since the main strategy to prove the existence of the iterates was modified to be adapt in others scenarios in the future that do not involve point-set operators.  Our motivation for presenting this version of the algorithm is to introduce a more adaptable version of the method for typical generalizations.

Besides proving the convergence, we study the complexity bounds of the algorithm, as is commonly analysed for subgradients methods. Subgradient methods can also be used to optimize weakly convex functions, but when additional structure is present in the function, better complexity can be achieved. In particular, if the proximal operator of a nonsmooth weakly convex function can be computed analytically, it is possible to prove complexity bounds of $\mathcal{O}\left(\epsilon^{-2}\right)$, which is the same as for steepest descent methods in the smooth nonconvex case. This indicates that the nonsmoothness can be mitigated if the proximal operator can be applied to the function.

For convex optimization problems, gradient methods on smooth functions can achieve complexity bounds of $\mathcal{O}\left(\epsilon^{-1}\right)$, and accelerated gradient methods can achieve $\mathcal{O}\left(\epsilon^{-1/2}\right)$. These same bounds can be achieved for nonsmooth problems if the nonsmooth term is handled by a proximal operator. However, if the explicit proximal operator is not available and subgradient methods must be used, the complexity reverts to $\mathcal{O}\left(\epsilon^{-2}\right)$. Hence, by assuming that $f$ is weakly-convex and nonsmooth we show that our inexact proximal point method achieve the complexity $\mathcal{O}\left(\epsilon^{-2}\right).$ See more details about complexity of weakly-convex functions in \cite{Bohn}.

The organization of the paper is as follows. In Section \ref{section2}, some notation
and basic results used throughout the paper are presented. In Section \ref{section3}, we
introduce the very important existence lemma. This result is the key to our convergence analysis. Our inexact proximal method is presented in Section \ref{section4}, and the main results are stated and proved. Some final remarks are made in Section \ref{section5}.

\section{Preliminaries}\label{section2}

The main definitions and notations used throughout the paper are presented. We adopt the standard notation of non-smooth analysis as in \cite{Drusvyatskiy2019,Mordukhovich2006,Rockafellar1970,RockafellarWets1998}.

\subsection{Basic notation in variational analysis}
The Euclidean scalar product of $\mathbb{R}^n$ and its corresponding norm are respectively denoted by $\langle\cdot,\cdot\rangle$ and $\|\cdot\|$. If $\bar x\in \mathbb{R}^n$, the \textit{open ball} centered at $\bar{x}$ with radius $\delta$ is the set $B(\bar{x},\delta)$. Similarly, $B[\bar{x},\delta]$ will denote the \textit{closed ball} with center $\bar{x}$ with radius $\delta$. The \textit{effective domain} of a lower semicontinuous function (lsc) 
$f:\mathbb{R}^n\to\overline{\mathbb{R}}:={\mathbb{R}}\cup\{\pm \infty\}$, denoted by $\mbox{dom}~f$, is the subset of ${\mathbb{R}}^n$ on which $f$ is finite-valued, i.e., $\mbox{dom}\,f:=\{x\in{\mathbb{R}}^n:f(x)<+\infty\}$. 

The \textit{Moreau envelopes} $e_\lambda f$ of a given function $f:\mathbb{R}^n\to\overline{\mathbb{R}}$, for any $\lambda>0$, is defined by
\begin{equation*}\label{Moreau}
e_{\lambda}f(x):=\min_{y}\left\{f(y)+\frac{1}{2\lambda}\|y-x\|^2\right\},
\end{equation*}
and the associated \textit{proximal operator}, $P_{\lambda} f$, defined by
\begin{equation*}\label{proxope}
P_{\lambda}f(x):=\mbox{argmin}_{y}\left\{f(y)+\frac{1}{2\lambda}\|y-x\|^2\right\}.
\end{equation*}

Given a convex function $f:\mathbb{R}^n\to\overline{\mathbb{R}}$, a vector $v$ is called a subgradient of $f$ at a
point $x\in\mbox{dom} f$ if the inequality $f(y)\geq f(x)+\langle v,y-x\rangle$ holds for all $y\in\mathbb{R}^n$. The set of all subgradients of $f$ at $x$ is denoted by $\partial_C f(x)$, and is called the \textit{subdifferential} of $f$ at $x$. For any point $x\notin\mbox{dom} f$ , we define $\partial_C f (x)$ to be the empty set.	

Let $f$ be a proper and lower semicontinuous function and a point $\bar{x}$ with $f(\bar{x})$ finite. The \textit{Fréchet subdifferential} of $f$ at $\bar{x}$, denoted by $\hat{\partial}f(\bar{x})$, is the set of all vectors $v$ satisfying
\begin{equation*}
\displaystyle\liminf_{y\to \bar{x}; y\neq \bar{x}}
\frac{1}{\|y-\bar{x}\|}\left(f(y)-f(\bar{x})-\langle v,y-\bar{x} \rangle\right)\geq 0.
\end{equation*}
As noted by Bolte et al. in \cite{Bolte2007}, the Fréchet subdifferential is not completely satisfactory in optimization, since $\hat{\partial} f(x)$ might be empty-valued at points of particular 
interest. Moreover, the Fréchet subdifferential is not a closed mapping (see, for instance, \cite[p.304]{RockafellarWets1998}), a crucial assumption in convergence analysis. This justifies the necessity of an enlargement of the Fréchet subdifferential. To fill this gap, one defines the \textit{limiting subdifferential} of $f$ at $\bar{x}$, as the set of all vectors $v$ satisfying
\begin{equation*}
\partial f(\bar{x}):=
\left\{x^{\ast}\in\mathbb{R}^n~:~\exists x_{n}\rightarrow \bar{x}, f(x_n)\rightarrow
f(\bar{x}),x_{n}^{\ast}\in\hat{\partial}f(x_n);x_{n}^{\ast}\rightarrow
x^{\ast}\right\}.
\end{equation*}
It is straightforward to check that $\hat{\partial}f(x)\subset \partial f(x)$. For convex functions $f$, the subdifferentials $\hat{\partial}f(x)$, $\partial_C f(x)$ and $\partial f(x)$ coincide, while for $C^1$-smooth functions $f$, they consist only of the gradient $\nabla f (x)$.

We say that $\bar{x}$ is \textit{stationary} for $f$ if the inclusion $0\in\partial f (\bar{x})$ holds.

\subsection{Weak convexity}

\begin{definition} We say that a function $f:\mathbb{R}^n\to\overline{\mathbb{R}}$ is $\rho$-weakly
convex on a set $U$ if for any points $x, y\in U$ and $\alpha\in [0,1]$, the approximate secant
inequality holds:
$$
f(\alpha x+(1-\alpha)y)\leq \alpha f(x)+(1-\alpha)f(y)+\frac{\rho\alpha(1-\alpha)}{2}\|x-y\|^2.
$$
\end{definition}

It is well-known that for a locally Lipschitz function $f:\mathbb{R}^n\to\mathbb{R}$, the following statements
are equivalent; see e.g. \cite{Daniilidis2005}.

\begin{enumerate}
\item (Weak convexity) $f$ is $\rho$-weakly convex on $\mathbb{R}^n$ .
\item (Perturbed convexity) The function $f(\cdot)+\frac{\rho}{2}\|\cdot\|^2$ is convex on $\mathbb{R}^n$ .
\item (Quadratic lower-estimators) For any $x, y\in \mathbb{R}^n$ and $v\in\partial f (x)$, the inequality
\begin{equation*}\label{eqquadraticestimation}
f(y)\geq f(x)+\langle v,y-x\rangle-\frac{\rho}{2}\|y-x\|^2
\end{equation*} 
holds.
\end{enumerate}

Clearly, a convex function  is weakly convex, in this case $\rho=0.$ Also, if $f$ is a smooth function with a $L$-Lipschitz continuous gradient, then $f$ is weakly convex function with $\rho=L.$ Moreover, the composite of a convex with a weakly convex function is weakly convex, i.e., is $f(x)=g(h(x))$, when $g:\mathbb{R}^m\to \mathbb{R}$ convex and Lipschitz and $h:\mathbb{R}^n\to \mathbb{R}^m$ smooth with gradient Lipschitz, then $f$ is weakly convex. More details can be found in \cite{Drusvyatskiy2019,Oliveira,Syrtseva}.

An immediate consequence of the above definition is that any weakly convex function $f$ can be expressed as a difference of convex (DC) functions. Specifically, we can consider the function 
$\xi(x):=f(x)+\frac{\rho}{2}\|x\|^2$, 
which is convex since $f$ is weakly convex. This allows us to decompose $f$ as:
$$
f(x)=\xi(x)-\frac{\rho}{2}\|x\|^2,
$$  
where $\rho\geq 0$ is a constant. This is a natural DC decomposition of $f.$

We conclude this section by presenting the following result, that establishes the existence of a neighborhood in which the proximal operator, $P_{\lambda} f$, is well-defined and a single-valued, Lipschitz continuous operator.
\begin{proposition}\label{Mainprop}
\rm{(Moreau envelope of $\rho$-weakly convex function)} Assume that $\bar{x}\in\mbox{dom}\,f$ is such that $0\in\partial f(\bar{x})$. Suppose that $f$ is $\rho$-weakly convex on $\mathbb{R}^n$, $\rho>0$ and fix $\lambda\in]0, 1/\rho[$. Then, there exists $\alpha>0$ such that, on $B[\bar{x},\alpha]$, the proximal operator, $P_\lambda f$, is well-defined, single-valued, Lipschitz continuous with constant $1/(1-\lambda\rho)$ and
\begin{equation*}\label{Mainprop1}
P_\lambda f(x)=(I+\lambda \partial f)^{-1}(x).
\end{equation*}
Besides, the Moreau envelope $e_\lambda f$ is $C^1$-smooth with gradient
\begin{equation}\label{Mainprop2}
\nabla e_\lambda f(x)=\lambda^{-1}(x-P_\lambda f(x))=\lambda^{-1}\left[I-\left[I+\lambda\partial f\right]^{-1}\right](x).
\end{equation}
Moreover, stationary points of $e_\lambda f$ and of $f$ coincide.
\end{proposition}
\begin{proof}
This proposition is a combination of two other results that can be found in
\cite[Proposition 3.1 and Corollary 3.4] {Hoheisel2018} .
\end{proof}

\section{The existence result}\label{section3}

The problem of finding iterates for proximal algorithms when dealing with weakly convex functions is a complex one. For nonmonotone operators, the local maximal monotonicity of the so-called Yosida regularization was enough to address this issue in  \cite{IusemPennanen2003,Pennanen2002}. However, in our case, as we are not working with a point-set operator, we present a different proof of this result using a very classical contraction mapping principle, the Banach fixed-point theorem. This important theory result, which is used as primary support in proving convergence result, can be found in \cite[Theorem 4.48, p.92]{Apostol1974}. The result is the following:

\begin{theorem} \label{pointfixedtheorem} \rm{(Banach fixed-point theorem)} \\
Let $(X,d)$ be a nonempty complete metric space and let $\Phi:X\to X$ be a strict
contraction, i.e.,
$$
d(\Phi(x),\Phi(y))<\theta d(x,y)\quad\forall x,y\in X\quad\mbox{with}\quad 0<\theta<1.
$$
Then $\Phi$ has a unique fixed point, $\Phi(z)=z$.
\end{theorem}

As a consequence of the above result, we next establish the following result: 

\begin{lemma}\label{mainlemma} Let $f:\mathbb{R}^n\to\overline{\mathbb{R}}$ be a proper, bounded below and lower semicontinuous function at $\bar x\in\mbox{dom}\,f$. Suppose that $f$ is $\rho$-weakly convex on a set $U$ and $0\in\partial f(\bar x)$, suppose that $e_{\gamma}f $ is convex on $B[\bar x, \delta]$.
Take any $\lambda,\gamma>0$ satisfying, 
\begin{equation*}\label{maincity1}
0<2\gamma<\lambda<\frac{1}{\rho}.
\end{equation*}
Define
\begin{equation}\label{L}
L:=\left[1+\frac{(\lambda-\gamma)}{\gamma}\left(1+\frac{1}{1-\gamma \rho}\right)\right],
\end{equation}
and take any positive constant $\sigma$ such that $\sigma<2/L^2$. Then, if 
\begin{equation*}
\kappa:=1-2\sigma+\sigma^2L^2,
\end{equation*} we have $0<\kappa<1$.
Furthermore, take any $\beta$ such that
\begin{equation}\label{beta}
0<\beta<\min\left\{\delta,\frac{\delta}{\sigma}(1-\sqrt{\kappa})\right\}.
\end{equation}
Then, the following holds:
\begin{description}
\item[\rm{(i)}] For each $x\in B[\bar x,\beta]$, there exists 
$z\in B(\bar x, \delta)$ satisfying
\begin{equation}\label{eqdefinitionz}
z=x-(\lambda-\gamma)\nabla e_{\gamma}f(z).
\end{equation} 
\item[\rm{(ii)}] If we define
\begin{equation}\label{y}
y=z-\gamma(\lambda-\gamma)^{-1}(x-z),
\end{equation}
it follows that  $y\in B[\bar x, \beta]$ and $y = P_{\lambda}f(x)$.
\end{description} 
\end{lemma}

\begin{proof}
From (\ref{L}), it follows that $L>1$, and then $$0\leq(\sigma-1)^2<1-2\sigma+\sigma^2L^2=\kappa.$$ On the other, since $\sigma<2/L^2$, we have 
$$\kappa=1-2\sigma+\sigma^2L^2<1.$$ This proves that $0<\kappa<1$.

Let us prove (i). Since $f$ is  $\rho$-weakly convex on a set $U$ and $0\in\partial f(\bar x)$, it follows from Proposition \ref{Mainprop} that for each $\gamma\in]0, 1/\rho[$ there is a neighborhood $B[\bar{x},\alpha]$ such that the mapping $P_\gamma f$ is single-valued and Lipschitz continuous with
constant $1/(1-\gamma\rho)$, and the function $e_\gamma f$ is $C^1$-smooth in $B[\bar{x},\alpha]$. If necessary, shrink $\delta$ so that $\delta<\alpha$ and define
\begin{equation}\label{S}
S(z)=z+(\lambda-\gamma)\nabla e_{\gamma}f(z),\quad z\in B[\bar{x},\delta].
\end{equation}
For all $z,w\in B[\bar{x},\delta]$,
\begin{eqnarray*}
\langle S(z)-S(w),z-w\rangle
&=&\langle z+(\lambda-\gamma)\nabla e_{\gamma}f(z)-w-(\lambda-\gamma)\nabla e_{\gamma}f(w),z-w\rangle\\
&=&\langle (z-w)+(\lambda-\gamma)(\nabla e_{\gamma}f(z)-\nabla e_{\gamma}f(w)),z-w\rangle\\
&=&\|z-w\|^2+(\lambda-\gamma)\langle \nabla e_{\gamma}f(z)-\nabla e_{\gamma}f(w),z-w\rangle.
\end{eqnarray*}
Taking into account that $e_{\gamma}f$ is convex on $B[\bar{x},\delta]$, for all $z,w\in B[\bar{x},\delta]$,
$$
\langle \nabla e_{\gamma}f(z)-\nabla e_{\gamma}f(w),z-w\rangle\geq0.
$$
Hence it follows that, for all $z,w\in B[\bar{x},\delta]$,
\begin{equation}\label{ineq1}
\langle S(z)-S(w),z-w\rangle\geq \|z-w\|^2.
\end{equation}
On the other hand, for all $z,w\in B[\bar{x},\delta]$, we can use the well known triangle inequality to obtain that
\begin{eqnarray}\label{eq:BD}
\|S(z)-S(w)\|
&=&\|z+(\lambda-\gamma)\nabla e_{\gamma}f(z)-w-(\lambda-\gamma)\nabla e_{\gamma}f(w)\|\nonumber\\
&\leq &\|z-w\|+(\lambda-\gamma)\|\nabla e_{\gamma}f(z)-\nabla e_{\gamma}f(w)\|.
\end{eqnarray}
Taking into account that \eqref{Mainprop2} holds, we have that
\begin{eqnarray*}
\|\nabla e_{\gamma}f(z)-\nabla e_{\gamma}f(w)\|
&\leq&\frac{1}{\gamma}\|(z-P_\gamma f(z))-(w-P_\gamma f(w))\|\\
&\leq&\frac{1}{\gamma}\left(\|z-w\|+\|P_\gamma f(z)-P_\gamma f(w)\|\right).
\end{eqnarray*}
Due to Proposition \ref{Mainprop}, $P_{\gamma}f$ is Lipschitz continuous with constant $1/(1-\gamma\rho)$. Hence,
last inequality becomes
$$
\|\nabla e_{\gamma}f(z)-\nabla e_{\gamma}f(w)\|\leq \frac{1}{\gamma}\left(1+\frac{1}{1-\gamma \rho}\right)\|z-w\|.
$$
Combining last inequality with \eqref{eq:BD}, and using the definition of $L$ in \eqref{L}, we conclude that
\begin{eqnarray*}
\|S(z)-S(w)\|
&\leq&\|z-w\|+\frac{(\lambda-\gamma)}{\gamma}\left(1+\frac{1}{1-\gamma \rho}\right)\|z-w\|\\
&=& \left[1+\frac{(\lambda-\gamma)}{\gamma}\left(1+\frac{1}{1-\gamma \rho}\right)\right]\|z-w\|\\
&=& L\|z-w\|.
\end{eqnarray*}
Consequently, for all $z,w\in B[\bar{x},\delta]$, we have
\begin{equation}\label{ineq2}
\|S(z)-S(w)\| \leq  L\|z-w\|.
\end{equation}
Now, take any $x\in B[\bar{x},\beta]$ and define
\begin{equation}\label{phi}
\Phi(z)=\sigma x-\sigma S(z)+z,\quad z\in B[\bar{x},\delta].
\end{equation}
Next, we will verify that $\Phi$ is a strict contraction with $\theta=\sqrt{\kappa}$. For any $z,w\in B[\bar x, \delta]$, and  after simple manipulations combining \eqref{ineq1}, \eqref{ineq2} and \eqref{phi} we obtain 
\begin{eqnarray*}
\|\Phi(z)-\Phi(w)\|^2
&=&\|(\sigma x-\sigma S(z)+z)-(\sigma x-\sigma S(w)+w)\|^2\\
&=&\|(z-w)-\sigma(S(z)-S(w))\|^2\\
&=&\|z-w\|^2-2\sigma\langle z-w,S(z)-S(w)\rangle+\sigma^2\|S(z)-S(w)\|^2\\
&\leq &\|z-w\|^2-2\sigma\|z-w\|^2+\sigma^2L^2\|z-w\|^2\\
&=&\left(1-2\sigma+\sigma^2L^2\right)\|z-w\|^2.
\end{eqnarray*}
Therefore, 
$$
\|\Phi(z)-\Phi(w)\|\leq\sqrt{\kappa}\|z-w\|.
$$
This proves that $\Phi$ is a strict contraction. Therefore, we can apply  Theorem {pointfixedtheorem} obtaining that  $\Phi(z)$ has a unique fixed point $\Phi(z)=z$. From the definition of $S$, we can obtain that $ x=S(z)$. If remains to prove that $z\in B(\bar{x},\delta)$. In fact, since $0\in\partial f(\bar x)$, we have that $\bar{x}$ is also a stationary point of $e_{\gamma} f$, due to Proposition \ref{Mainprop}. Using the expression \eqref{S}, we can conclude that $\bar{x}=S(\bar{x})$. Thus,
$$
\|\bar{x}-\Phi(\bar{x})\|=\|\bar{x}-(\sigma x-\sigma \bar{x}+\bar{x})\|=\sigma\|x-\bar{x}\|<\sigma\beta<\delta(1-\sqrt{\kappa}).
$$
Since $z$ is a fixed point of $\Phi$, we have
\begin{eqnarray*}
\|z-\bar{x}\|
&\leq&\|\Phi(z)-\Phi(\bar{x})\|+\|\Phi(\bar{x})-\bar{x}\|\\
&<&\sqrt{\kappa}\|z-\bar{x}\|+\delta(1-\sqrt{\kappa}).
\end{eqnarray*}
Hence, $(1-\sqrt{\kappa})\|z-\bar{x}\|<\delta(1-\sqrt{\kappa})$. This shows that $z\in B(\bar{x},\delta)$, completing the proof of \eqref{eqdefinitionz}. 

To prove (ii), consider $y$ defined as  in \eqref{y}. Then we are going to prove that $y$ lies in the closed ball  $ B[\bar x, \delta]$  and $y=P_{\lambda}f(x)$.
To do so, initially, observe that
\begin{equation}\label{maineq}
(\lambda-\gamma)^{-1}(x-z)={\gamma}^{-1}(z-y)=\lambda^{-1}(x-y),
\end{equation}
and take any $\bar x\in U$ with $0\in\partial f(\bar x)$. Again, Proposition \ref{Mainprop} guarantees that $\nabla e_{\gamma}f(\bar{x})=0$. Hence, after straightforward manipulations,  we obtain
\begin{eqnarray*}
\|x-\bar x\|^2
&=&\|x-z\|^2+\|z-\bar x\|^2+2\langle x-z,z-\bar x\rangle\\
&=&\|x-z\|^2+\|z-\bar x\|^2+2(\lambda-\gamma)\langle\nabla e_{\gamma}f(z),z-\bar x\rangle\\
&\geq&\|x-z\|^2+\|z-\bar x\|^2,
\end{eqnarray*}
where the last inequality holds owing to the fact that $e_{\gamma}$ is convex on  $B[\bar x, \delta]$. On the other hand
\begin{eqnarray*}
\|y-\bar x\|^2
&=&\|y-z\|^2+\|z-\bar x\|^2+2\langle y-z,z-\bar x\rangle\\
&=&\left(\frac{\gamma}{\lambda-\gamma}\right)^2\|x-z\|^2+\|z-\bar x\|^2+\frac{2\gamma}{(\lambda-\gamma)}\langle z-x,z-\bar x\rangle\\
&\leq&\left(\frac{\gamma}{\lambda-\gamma}\right)^2\|x-z\|^2+\|x-\bar x\|^2-\|x-z\|^2-2{\gamma}\langle\nabla 
e_{\gamma}f(z),z-\bar x\rangle\\
&\leq&\|x-\bar x\|^2-\left(1-\left(\frac{\gamma}{\lambda-\gamma}\right)^2\right)\|x-z\|^2\\
&\leq&\|x-\bar{x}\|^2,
\end{eqnarray*}
where the last inequality holds since  $2\gamma<\lambda$. Hence, it follows easily that $y\in B[\bar x, \delta]$. To finish of the lemma, we just have to conclude that $y=P_{\lambda}f(x)$. From Proposition \ref{Mainprop}, we have that 
$$\nabla e_\gamma f(x)=\lambda^{-1}\left[I-\left[I+\gamma\partial f\right]^{-1}\right](x).$$ Thanks to 
\cite[Lemma 4.5]{Rockafellar1996}, the previous expression can be rewritten as
$$
\nabla e_\gamma f(x)= (\gamma I+[\partial f]^{-1})^{-1}(x).
$$
 In view of \eqref{eqdefinitionz}, we have
$$
(\lambda-\gamma)^{-1}(x-z)=(\gamma I+[\partial f]^{-1})^{-1}(z).
$$
Consequently,
$$
z\in(\gamma I+[\partial f]^{-1})((\lambda-\gamma)^{-1}(x-z)).
$$
Hence, it is easy to obtain that
$$
z-\gamma(\lambda-\gamma)^{-1}(x-z)\in[\partial f]^{-1}(\lambda-\gamma)^{-1}(x-z)).
$$
Since $y= z-\gamma(\lambda-\gamma)^{-1}(x-z)$, we can use \eqref{maineq} to obtain
$$
y\in[\partial f]^{-1}(\lambda-\gamma)^{-1}(x-z).
$$
Thus,
\begin{eqnarray*}
y\in[\partial f]^{-1}((\lambda-\gamma)^{-1}(x-z))&\Longleftrightarrow& (\lambda-\gamma)^{-1}(x-z)\in\partial f(y).\\
&\Longleftrightarrow& \lambda^{-1}(x-y)\in\partial f(y)\\
&\Longleftrightarrow& x\in(I+\lambda \partial f)(y)\\
&\Longleftrightarrow& y\in(I+\lambda \partial f)^{-1}(x).
\end{eqnarray*}
Due to Proposition \ref{Mainprop}, the proximal operator is well-defined, single-valued, and $P_{\lambda}f(x)=(I+\lambda \partial f)^{-1}(x)$ holds for all $x\in B[\bar x,\beta]$. Therefore, $y=P_{\lambda}f(x)$, which completes the proof.
\end{proof}

\section{The algorithm and convergence analysis}\label{section4}

Now, we present a version of the inexact proximal method considered in \cite{IusemPennanen2003} for weakly convex functions. From now on, $f:\mathbb{R}^n\to\overline{\mathbb{R}}$ is a proper $\rho$-weakly convex on a set $U$, bounded below and lower semicontinuous function on $U$ and let $\bar x\in\mbox{dom}\,f\cap U$ be a point such that $0\in\partial f(\bar x)$. Furthermore, the following assumption is considered throughout the entire work:
\begin{description}
\item[\sc{Assumption 1:}] For $\delta>0$, the Moreau envelope of $f$ is convex on $B[\bar x, \delta]$.
\end{description}

Let $\bar{\lambda}$ be a positive parameter. The method, to be called \textbf{Algorithm 1}, requires bounded sequences of positive real numbers $\{\gamma_k\},\{\lambda_k\}$ satisfying, 
\begin{equation*}\label{maincity23}
0<\bar{\lambda}<2\gamma_k<\lambda_k<\dfrac{1}{\rho},\quad\forall k\geq0.
\end{equation*}
In addition, let $\beta>0$ to be taken as in \eqref{beta}. The algorithm is presented in the following way:\\
---------------------------------------------------------------------------------------------------------------------\\
 \begin{algorithm1} \\
---------------------------------------------------------------------------------------------------------------------\\
\begin{description}
\item[\sc{Initialization:}] Choose $x^0\in B[\bar{x},\beta]$ as the inicial point. 
\item[\sc{Iterative step:}] 
\item[(a)] Given $x^k\in B[\bar{x},\beta]$, compute $z^k\in B(\bar{x},\delta)$, satisfying  
\begin{equation}\label{definitionz}
z^k=x^k-(\lambda_k-\gamma_k)\nabla e_{\gamma_k}f(z^k).
\end{equation}
\item[(b)] The iterate $x^{k+1}$ is given by  
\begin{equation}\label{definitiony}
x^{k+1}=z^k-\gamma_{k}(\lambda_k-\gamma_k)^{-1}(x^k-z^k).
\end{equation}
\end{description}
\end{algorithm1}
-----------------------------------------------------------------------------------------------------------------------------------------

Next, we prove that Algorithm 1 is well defined. 

\begin{proposition}\label{welldefinedness} The Algorithm 1 is well defined. Besides, for all $k\geq0$,
$$x^{k+1}=P_{\lambda_k}(x^k)\quad\mbox{and}\quad x^k\in B[\bar{x},\beta].$$
In addition, we have that $\sum_k\|x^k-x^{k+1}\|^2<\infty$. 
\end{proposition}
\begin{proof}
Take $x^0 \in B[\bar x,\beta] $ and assume that $x^k\in B[\bar x,\beta]$ hold for some $k$. We will prove that 
 $x^{k+1}\in B[\bar x,\beta]$. Applying Lemma \ref{mainlemma} for $x=x^k$, $z=z^{k}$ and $y=x^{k+1}$, it yields that there exists $z^k\in B[\bar x, \delta]$  satisfying \eqref{definitionz} and \eqref{definitiony}. This proves the well definedness of the Algorithm 1. In addition, we obtain that, for all $k\geq0$, $x^{k+1}\in B[\bar x, \beta]$ and $x^{k+1} = P_{\lambda_k}f(x^k)$. On the other hand, we can use the definition of $P_{\lambda_k}f$ to obtain that
\begin{equation}\label{ineqf(x)}
\frac{1}{2\lambda_k}\|x^k-x^{k+1}\|^2\leq f(x^k)-f(x^{k+1}).
\end{equation}
Summing the last inequality over $k=1,\dots,n$, for all $k\geq0$, and using that $\lambda_k<1/\rho$,  we have 
$$\sum_{k=1}^{n}\|x^k-x^{k+1}\|^2\leq\frac{2}{\rho}( f(x^1)-f(x^{n+1})).$$ Thanks to \eqref{ineqf(x)}, $\{f(x^k)\}$ is non-increasing. Since $f$ is bounded below we can that $\{f(x^k)\}$ is bounded. Thus, it is easy to prove that $\sum_k\|x^k-x^{k+1}\|^2$ converges.
\end{proof}

Thereafter, we established a very useful lemma. 

\begin{lemma} For all $x\in  B[\bar x,\beta]$,
\begin{equation}\label{eq:mainfejer}
		\left\|x^{k+1}-x\right\|^2\leq\left\|x^{k}-x\right\|^2+2\lambda_k\left(f(x)-f(x^{k+1})\right)
		,\quad k=0,1,\dots.
		\end{equation}
\end{lemma}
\begin{proof}
Since $x^{k+1} = P_{\lambda_k}f(x^k)$, we can use \eqref{Mainprop2} to obtain that
$$
\nabla e_{\lambda_k}f(x^k)=\frac{1}{\lambda_k}(x^k-x^{k+1}).
$$
Now, take any $x\in  B[\bar x,\beta]$. Once $e_{\lambda_k} f$ is a convex function on $B[\bar x,\beta]$,
$$
e_{\lambda_k}f(x)\geq e_{\lambda_k}f(x^k)+\left\langle \nabla e_{\lambda_k}f(x^k),x-x^k\right\rangle.
$$
Again, taking into account that $x^{k+1} = P_{\lambda_k}f(x^k)$ holds 
$$
e_{\lambda_k}f(x^k)=f(x^{k+1})+\frac{1}{2\lambda_k}\left\|x^{k+1}-x^k\right\|^2.
$$
Besides, considering that $f(x)\geq e_{\lambda_k}f(x)$, we obtain 
\begin{equation}\label{eqtext}
f(x)-f(x^{k+1})\geq\frac{1}{2\lambda_k}\left\|x^{k+1}-x^k\right\|^2+\frac{1}{\lambda_k}\left\langle x^k-x^{k+1},x-x^k\right\rangle.
\end{equation}
Hence, since
$$
\left\|x^{k+1}-x\right\|^2=\left\|x^{k+1}-x^k\right\|^2+\left\|x^{k}-x\right\|^2+2\left\langle x^k-x^{k+1},x-x^k\right\rangle,
$$
we can combine the previous expression with \eqref{eqtext} to obtain \eqref{eq:mainfejer}.

\end{proof}

Now, we will briefly review two essential lemmas on non-negative sequences. Subsequently, we will introduce our main convergence results, building on this groundwork.

\begin{lemma}[see \cite{Polyak1987}]\label{lemmaone} Let $\{u^k\}$, $\{\alpha_k\}$, and $\{\beta_k\}$ be nonnegative sequences of real numbers satisfying $u^{k+1}\leq(1+\alpha_k)u^k+\beta_k$ such that
	$\sum_k\alpha_k<\infty$ and $\sum_k\beta_k<\infty$. Then, the sequence $\{u^k\}$ converges.
\end{lemma}

\begin{lemma}[see \cite{Polyak1987}]\label{lemmasec}
	Let $\{\lambda_k\}$ be a sequence of positive numbers, $\{a_k\}$ a sequence of real numbers, and $b_n:=\sigma^{-1}_n\sum_{k=1}^n\lambda_ka_k$, 
	where $\sigma_n:=\sum_{k=1}^n\lambda_k$. If $\sigma_n\to\infty$, $\liminf a_n\leq\liminf b_n\leq\limsup b_n\leq\limsup a_n$.
\end{lemma}

Our main result will now be presented.
\begin{theorem}\label{maintheorem}  
Let $\{x^k\}$ be the sequence generated by Algorithm 1. Then, $\{f(x^k)\}$ converges to $f^*$, where $f^*:=\inf_{x\in B[\bar{x},\beta]} f(x)$. In addition, $\{x^k\}$ converges to $x^*$ with $f^*=f(x^*)$.
\end{theorem}

\begin{proof}
From Proposition \ref{welldefinedness}, for all $k\geq0$, the sequence $\{x^k\}$ lies on the closed ball $B[\bar{x},\beta]$. Hence, the sequence is bounded. Thus, let $x^*$ be an accumulation point of $\{x^k\}$ and $\{x^{k_j}\}$ a subsequence of $\{x^k\}$ such that $\lim_{j\to+\infty}x^{k_j}=x^*$. As $f$ is lower semicontinuous, 
\begin{equation}\label{France9}
f^*\leq f(x^*)\leq\liminf_{j\to+\infty} f(x^{k_j}).
\end{equation}
Due to Proposition \ref{welldefinedness}, the sequence $\{x^k-x^{k+1}\}$ converges to zero, since $\sum_k\|x^k-x^{k+1}\|^2<\infty$. Consequently, $$\liminf_{j\to+\infty} f(x^{k_j})=\liminf_{j\to+\infty} f(x^{k_j+1}).$$ 
Setting $x=x^{k}$ in \eqref{eq:mainfejer}, we obtain
$$
f(x^{k+1})\leq f(x^k).
$$
Since $x^k\in B[\bar{x},\beta]$,  $f^*\leq f(x^k)$ for all $k\geq0$. The, we have
$$
0\leq f(x^{k+1})-f^*\leq f(x^k)-f^*.
$$
Then using Lemma \ref{lemmaone} it follows that the sequence $\{f(x^k)-f^*\}$ converges.
On the other hand, for all $x\in B[\bar{x},\beta]$, we can deduce from \eqref{eq:mainfejer} that
$$
2\lambda_k\left(f(x^{k+1})-f(x)\right)\leq\left\|x^{k}-x\right\|^2-\left\|x^{k+1}-x\right\|^2.
$$
Summing the last inequality over $k=1,\dots,n$ we obtain 
$$
\sum_{k=1}^{n}2\lambda_k\left(f(x^{k+1})-f(x)\right)\leq\left\|x^{1}-x\right\|^2-\left\|x^{n+1}-x\right\|^2.
$$
Hence,
$$
\sigma_n^{-1}\sum_{k=1}^{n}2\lambda_k f(x^{k+1})\leq f(x)+\sigma_n^{-1}\left\|x^{1}-x\right\|^2,
$$
where $\sigma_n:=\sum_{k=1}^{n}2\lambda_k$. As $\lambda_k>\bar{\lambda}$, then $\sigma_n\to\infty$. Thus, we can use 
Lemma \ref{lemmasec} to conclude that,
$$
\liminf_{n\to\infty}f(x^{k+1})\leq\limsup_{n\to\infty}\sigma_n^{-1}\sum_{k=1}^{n}2\lambda_k f(x^{k+1})\leq f(x),
$$
for all $x\in B[\bar{x},\beta]$. Then, since $f$ is lower semicontinuous, last inequality implies that
\begin{equation}\label{Italy9}
f(x^*)\leq\liminf_{n\to\infty}f(x^{k_j+1})\leq \inf_{x\in B[\bar{x},\beta]} f(x)=f^*.
\end{equation}
Combining \eqref{France9} with \eqref{Italy9}, we have that $\liminf_{j\to+\infty} f(x^{k_j})=f^*=f(x^*)$.  Considering that  $\{f(x^k)-f^*\}$ is convergent, it must converge to zero. As, $f^*=f(x^*)$, we have that $f(x^*)<f(x^k)$ for all $k$. Setting $x=x^*$ in \eqref{eq:mainfejer}, we obtain $$\|x^{k+1}-x^*\|\leq\|x^k-x^*\|.$$ From Lemma \ref{lemmaone} we have that the sequence $\{\|x^k-x^*\|\}$ is convergent. Since $\lim_{j\to\infty}x^{k_j}=x^*$, we can conclude that $\{x^k\}$ converge to $x^*$.
	\end{proof}

	Assumption 1 plays a crucial role in our convergence analysis, making it the main idea explored in this paper. Given its significance, we will make a fell comments about the scenarios in which this hypothesis holds true. 

\begin{remark}
It is well known that whenever $f$ is proper, lower semicontinuous, and convex on an open $U\subset\mathbb{R}^n$, the proximal mappings $P_{\lambda}f$ are $1$-Lipschitz continuous on $U$. Besides,
$\nabla e_\lambda f(x)=\lambda^{-1}(x-P_\lambda f(x))$ for all $x\in U$. Hence, for any $x,y\in U$,
\begin{eqnarray*}
\langle\nabla e_{\lambda}f(z)-\nabla e_{\lambda}f(w),z-w\rangle
&=&\lambda^{-1}\langle (z-P_\lambda f(z))-(w-P_\lambda f(w)),z-w\rangle\\
&=&\lambda^{-1}\left(\|z-w\|^2-\langle P_\lambda f(z)-P_\lambda f(w),z-w\rangle\right)\\
&\geq &\lambda^{-1}\left(\|z-w\|^2-\|P_\lambda f(z)-P_\lambda f(w)\|\|z-w\|\right)\\
&\geq&0,
\end{eqnarray*}
where the last inequality is due to the $1$-Lipschitz continuity of $P_{\lambda}f$ on $U$. This shows that 
$e_{\lambda}f$ is also convex on $U$. This is the most obvious situation when the local convexity of the Moreau envelope holds true, see for instance \cite[Proposition 12.30]{Combettes}.
\end{remark}

Unlike the previous remark, the following example shows that the Moreau envelope of a weakly convex function could be locally convex even when its associated function is not.
\begin{example}\label{example1}
Consider $f:\mathbb{R}\to\mathbb{R}$ defined by
$$
f(x):=\max\{2-x^2,x^2\}.
$$
It is straightforward to check that $f$ is $\rho$-weakly convex for all $\rho\geq2$. On the other hand,
$0\in\partial f(1)=[-2,2]$. Hence, for all $x\in[1-2\lambda,1+2\lambda]$, we have $\lambda^{-1}(x-1)\in\partial f(1)$. Consequently, $1\in(I+\lambda \partial f)^{-1}(x)$. From Proposition \ref{Mainprop}, we obtain that, for all $x\in B[1,\alpha]$,  $P_{\lambda}f(x)=(I+\lambda \partial f)^{-1}(x)$. Taking $\lambda\leq\alpha/2$, we conclude that, for all $x\in B[1,2\lambda]$, $$P_\lambda(x)=1\quad\mbox{and}\quad e_\lambda f(x)=1+(1/2\lambda)(x-1)^2.$$ 
\end{example}

\begin{remark}
The local convexity of the Moreau envelope also occurs when $[\partial f]^{-1}$ has a Lipschitz localization.
The mapping $[\partial f]^{-1}:\mathbb{R}^n\rightrightarrows \mathbb{R}^n$ is said to have a \textit{Lipschitz localization} at a point $(x,y)\in \mbox{gph} [\partial f]^{-1}:=\{(x,y):v\in[\partial f]^{-1}(x)\}$ if there are neighborhoods $X\ni x$ and $Y\ni v$ such that the mapping 
$x\mapsto [\partial f]^{-1}(x)\cap Y$ is singlevalued and Lipschitz continuous on $X$, see  \cite{Pennanen2002}.   Extensive research in the literature has focused on investigating the existence of such localizations, and specific conditions have been identified to ensure their presence across various problem classes. For detailed insights, see notable works such as \cite{DontchevRockafellar1996,DontchevRockafellar1998,DontchevRockafellar2001,Klatte1999,Klatte1990,Levy2000,Robinson1980}, and the references therein. and other relevant references provide comprehensive information. Particularly, in 
\cite{Robinson1980} was given a sufficient condition known as strong regularity, which guarantees the existence of a Lipschitz localization for the solution mapping of a variational inequality. In addition, as a consequence of \cite[Proposition 7]{Pennanen2002}, if $[\partial f]^{-1}$ has a Lipschitz localization at $(0,\bar{x})$ with Lipschitz constant $\kappa$, then for $\lambda\geq\kappa$, $(\lambda I+[\partial f]^{-1})^{-1}$ is maximal monotone in a neighbourhood of $(0,\bar{x})$. This fact along with Proposition \ref{Mainprop}, shows that $e_{\lambda}f$ is locally convex around a critical point $\bar{x}$ if $[\partial f]^{-1}$ has a Lipschitz localization at $(0,\bar{x})$. Local convexity of $e_{\lambda}f$ was also investigated in \cite{Rockafellar1996}.
\end{remark}
	
We conclude this section by presenting a complexity result for the proposed algorithm. Before that, we will establish an auxiliary result.
\begin{lemma}\label{aux.lemma}
Given $\tau,$ $\lambda>0$ and a set $\left\{z^{j}\right\}_{j=1}^{k}$ of nonnegative real numbers, with $k\geq 2$, let
\begin{equation*}
	m(k):=\arg\min_{j\in\left\{1,\ldots,k-1\right\}}\left((z^{j})^{\tau}+(z^{j+1})^{\tau}\right).
	\label{eq:2.21}
\end{equation*}
If 
\begin{equation*}
	\sum_{j=1}^{k}(z^{j})^{\tau}\leq\lambda,
	\label{eq:2.22}
\end{equation*}
then
\begin{equation*}
	\max\left\{z^{m(k)},z^{m(k)+1}\right\}\leq\left(\dfrac{2\lambda}{k-1}\right)^{\frac{1}{\tau}}.
	\label{eq:2.23}
\end{equation*}
\end{lemma}
\begin{proof}
See \cite[Lemma~4]{GrapiglaSilva}.
\end{proof}
The following theorem establishes an iteration-complexity bound of $\mathcal{O}\left(\epsilon^{-2}\right)$ for the proposed method.
\begin{theorem}
	\label{thm:3.1}
	Given $\epsilon>0$, let $\left\{x^{k}\right\}_{k=1}^{T}$ be generated by the Algorithm 1, such that
	\begin{equation}
		\|x^{k+1}-x^k\|>\epsilon,\quad k=0,\ldots,T.
		\label{eq:3.10}
	\end{equation}
	Then,
	\begin{equation}
		T< 2+\frac{2}{\rho}(f(x^0)-f(x^*))\epsilon^{-2}.
		\label{eq:3.11}
	\end{equation}
\end{theorem}
\begin{proof}
It arises immediately from \eqref{eq:mainfejer} that
\begin{eqnarray}\label{eq:comple}
\sum_{j=0}^{k}\|x^{j+1}-x^j\|^2&\leq& \frac{2}{\rho}\sum_{j=0}^{k}(f(x^j)-f(x^{j+1}))=\frac{2}{\rho}(f(x^0)-f(x^{j+1}))\nonumber\\
&\leq& \frac{2}{\rho}(f(x^0)-f(x^*)),
\end{eqnarray}
where the last inequality is due $x^*$ be a local minimal of $f.$ Let us denote $s^k:=\|x^{k+1}-x^k\|.$ Let
\begin{equation*}
	k_{*}=\arg\min_{j\in\left\{1,\ldots,T-2\right\}}\left(\|s^{j}\|^{2}+\|s^{j+1}\|^{2}\right)
\end{equation*}
and assume that $T\geq 3$. Then, by Lemma \ref{aux.lemma} with $z^{j}=\|s^{j}\|$, $k=T-1$ and $\tau=2$, it follows from \eqref{eq:comple} that
\begin{equation}
	\max\left\{\|s^{k_{*}}\|,\|s^{k_{*}+1}\|\right\}\leq\left[\dfrac{2(f(x^{0})-f(x^*))}{\rho}\right]^{\frac{1}{2}}\dfrac{1}{(T-2)^{\frac{1}{2}}}.
	\label{eq:3.14}
\end{equation}
By combining \eqref{eq:3.10} with the previous inequality we can conclude that
$$
\epsilon<\max\left\{\|s^{k_{*}}\|,\|s^{k_{*}+1}\|\right\}\leq\left[\dfrac{2(f(x^{0})-f(x^*))}{\rho}\right]^{\frac{1}{2}}\dfrac{1}{(T-2)^{\frac{1}{2}}},
$$ 
which, after simples manipulations, implies the desired inequality \eqref{eq:3.11}.
\end{proof}

To summarize, the aforementioned theorem establishes that the algorithm analyzed in this paper has a worst-case evaluation complexity of no more than $\mathcal{O}\left(\epsilon^{-2}\right)$ for generating an iterate $x^k$ such that $\|x^{k+1}-x^k\|\leq\epsilon$, where $\epsilon>0$.

\section{Conclusions}\label{section5}

A proximal point algorithm to be applied to weakly convex functions was introduced. In contrast to previous work, we use the Moreau envelope as a replacement for the Yosida regularization. While our algorithm is similar to the previous versions, it offers more versatility for future generalizations, including those in scalar optimization or even in a multi-objective setting. By presenting this version of the algorithm, we hope to provide a more adaptable tool for addressing a wider range of optimization problems.

One of the primary challenges we face in our analysis is demonstrating the existence of the iterates of the method. To address this issue, our approach relies on the well-known Banach fixed-point theorem. This not only provides a more direct proof of the result, but it also sheds new light on the theoretical underpinnings of the algorithm. Additionally, this lemma will be better suited for future generalizations that do not involve point-set operators. 

Our analysis also includes a study of the complexity bounds of the algorithm. We derive a result that demonstrate the effectiveness of the method in terms of convergence rate and computational efficiency. 

Moreover, our work offers significant potential for future research in the field of optimization. By introducing a more versatile version of the proximal point algorithm, we open up new avenues for exploring a wider range of optimization problems. We believe that our approach can serve as a valuable starting point for addressing a variety of open questions in this area.

\end{document}